\documentclass[12pt,a4paper,reqno]{amsart}
\usepackage{amssymb,latexsym,amsfonts,amsthm,upref,amsmath,capt-of}
\usepackage[english]{babel}
\usepackage[margin=1in]{geometry} 
\usepackage[foot]{amsaddr}
\usepackage[dvipsnames,x11names]{xcolor}
\usepackage{hyperref}   
\hypersetup{colorlinks,citecolor=green,filecolor=black,linkcolor=blue,urlcolor=blue}
\usepackage{comment} 
\usepackage{enumerate} 
\usepackage{tikz}
\usetikzlibrary{chains}
\usepackage{float}
\usepackage{cleveref}
\usepackage{mathtools}
\usepackage{cite}
\usepackage{filecontents}
\usepackage{bbm}

\makeatletter

\newcommand{\leqnomode}{\tagsleft@true}
\newcommand{\reqnomode}{\tagsleft@false}

\makeatother

\makeatletter
\def\author@andify{%
	\nxandlist {\unskip ,\penalty-1 \space\ignorespaces}%
	{\unskip {} \@@and~}%
	{\unskip \penalty-2 \space \@@and~}%
}
\makeatother

\newcommand{\ta}[1]{\tau_{\lambda_{1},\theta_{1}}}

\newtheorem{thm}{Theorem}

\newtheorem{lem}{Lemma}
\numberwithin{equation}{section}
\numberwithin{prop}{section}
\numberwithin{lem}{section}
\numberwithin{thm}{section}

\numberwithin{cor}{section}

\theoremstyle{definition}
\newtheorem{defn}{Definition}
\numberwithin{defn}{section}

\newtheorem{example}{Example}

\usepackage{mathtools}
\usepackage{tikz}
\usetikzlibrary{chains}
\usepackage{graphics}
\usepackage{epic,eepic}
\usepackage{cite}

\newtheorem{rem}{Remark}
\numberwithin{rem}{section}

\newcommand{\beq}{\begin{equation}}
\newcommand{\eeq}{\end{equation}}
\newcommand{\beqn}{\begin{equation*}}
\newcommand{\eeqn}{\end{equation*}}

\newcommand{\g}{\mathfrak{g}}
\newcommand{\wtg}{\widetilde{\mathfrak{g}}}

\title[]{Quasi-particles and the Kanade--Russell and Kur\c sung\"{o}z formula for Capparelli's identity}

\author{Marijana Butorac}
\author{Slaven Ko\v{z}i\'{c}}
\author{Mirko Primc} 

\address[M. Butorac]{Faculty of Mathematics, University of Rijeka, Radmile Matej\v{c}i\'{c} 2, 51\,000 Rijeka, Croatia}
\email{mbutorac@math.uniri.hr}
\address[S. Ko\v{z}i\'{c} and M. Primc]{Department of Mathematics, Faculty of Science, University of Zagreb,  Bijeni\v{c}ka cesta 30, 10\,000 Zagreb, Croatia}
\email{kslaven@math.hr}
\email{primc@math.hr}



\begin{document}

\begin{abstract}
We construct a quasi-particle basis of the integrable highest weight module of highest weight $3\Lambda_0$ for the twisted affine Lie algebra of type $A_2^{(2)}$ in the principal realization. More specifically, by introducing the concept of polychromatic quasi-particle and finding relations among quasi-particles, we construct the spanning set of the standard module. Finally, its linear independence is proved by using Kanade--Russell and Kur\c sung\"{o}z's Andrews--Gordon type series of Capparelli's identities.
\end{abstract}

\maketitle

\numberwithin{equation}{section}

\tableofcontents

\frenchspacing

\section*{Introduction}
The works of J. Lepowsky and S. Milne \cite{LM} and J. Lepowsky and R. L. Wilson \cite{LW84} initiated an intensive research program on explicit realizations of the generalized Rogers--Ramanujan identities for standard affine Lie algebra modules by using the theory of vertex operator algebras (see in particular \cite{Capp1,Capp2,Capp3,CMP,KP,MP,N,PS,PT,TX,T}).
As a part of this program, by constructing a twisted vertex operator bases of level three standard $A_2^{(2)}$-modules in the principal realization parametrized by partitions satisfying difference two  condition, S. Capparelli discovered two new combinatorial partition identities of Rogers--Ramanujan-type (cf. \cite{Capp2,Capp3}), which were first proved in work \cite{A1} (see also  \cite{TX}). The first of these identities can be stated as follows.

{\em The number of partitions of a positive integer $n$ into distinct parts which are not congruent to $\pm 1 \ (\text{mod} \ 6)$ is equal to the number of partitions of $n$ into parts that are at least two, with difference at least two, and difference at least four unless the sum of successive parts is a multiple of three.}

The analytic form of the first Capparelli identity is
\beqn
\frac{1}{(q,q^5;q^6)(q^2,q^3,q^9,q^{10};q^{12})} =\frac{1}{(q,q^5;q^6)}\cdot \left. \sum_{n_1,n_2\geq 0} \frac{q^{2n_1^2+6n_1n_2+6n_2^2}x^{n_1+2n_2}}{(q;q)_{n_1}(q^3;q^3)_{n_2}} \ \right|_{x=1},
\eeqn  
where, the left-hand side (that is, the generating function of partitions with congruence restrictions on parts) is obtained from the principal specialization of the Weyl--Kac character formula (cf. \cite{Lep1}, \cite{K}), while the right-hand side, the Andrews--Gordon type (positive) generating function  of partitions with difference restrictions on parts, was found independently by S. Kanade and M. C. Russell in \cite{KR} and by K. Kur\c sung\"{o}z in \cite{Kur}, (see also \cite{Ka}). 

Using the methods of Capparelli, D. Nandi in \cite{N} conjectured three new Rogers--Ramanujan-type identities from the study of level four standard $A_2^{(2)}$-modules, which were proved by M. Takigiku and S. Tsuchioka in \cite{TT}. More recently, S. Capparelli, A. Meurman and M. Primc in \cite{CMP}  studied the standard $A_2^{(2)}$-module of  highest weight $5\Lambda_0$. 

In this paper, we study the standard $A_2^{(2)}$-module with highest weight $3\Lambda_0$ in the principal realization. Motivated by the works of A. Meurman and M. Primc, \cite{MP}, and S. Ko\v zi\' c and M. Primc, \cite{KP}, we generalize the notion of monochromatic quasi-particle (see also works of B. L. Feigin and A. V. Stoyanovsky, \cite{FS}, and G. Georgiev \cite{G}) to polychromatic quasi-particle. In particular, for positive roots $\alpha=\alpha_1$ and $\beta=\alpha_1+\alpha_2$, we define the quasi-particle $X(\alpha, \beta;n)$ of charge two and energy $n$ as 
\beqn
X^{(2)} (  n )=\text{Res}_{\zeta}\left(\zeta^{-n-1}X\left(\alpha,\beta; \zeta \right)\right),
\eeqn
where
\beqn
X\left(\alpha, \beta; \zeta\right)=\displaystyle\text{lim}_{\zeta_1,\zeta_2 \rightarrow \zeta}\left(\left(1+\frac{\zeta_1}{\zeta_2}\right)\left(1-\omega^2\frac{\zeta_1}{\zeta_2}\right)^2\left(1-\omega\frac{\zeta_1}{\zeta_2}\right)X\left(\alpha ; \zeta_1 \right) X\left(\beta ; \zeta_2 \right)\right),
\eeqn
and  
$X\left(\alpha ; \zeta_1 \right)=\sum X^{(1)}(j)\zeta_1^{j}$ and $ X\left(\beta ; \zeta_2 \right)$ 
are twisted vertex operators associated to $\alpha$ and $\beta$, respectively. 
Our main result states that the vectors  
\beqn
\alpha(i_1)\cdots \alpha(i_r)\, 
X^{(1)}(j_1)\cdots X^{(1)}(j_s)\,
X^{(2)}(k_1)\cdots X^{(2)}(k_t)\, v_{3\Lambda_0},
\eeqn
where
$i_1\leq i_2\leq \ldots \leq i_r< 0$, $i_p \equiv \pm 1\mod 6$, $j_1\leq j_2\leq \ldots \leq j_s< 0$, and $k_1\leq k_2\leq \ldots \leq k_t< 0$,  $k_p \equiv 0\mod 3$, such that the difference conditions  
$$j_p\leq j_{p+1}-4,\quad k_p\leq k_{p+1}-12,$$
and initial conditions 
$$j_s\leq -2-6t,\quad k_t\leq -6$$ 
hold, form a basis for the $A_2^{(2)}$-module   $L(3\Lambda_0)$. In order to find difference conditions and show that the given vectors span $L(3\Lambda_0)$, we derive and investigate relations between twisted vertex operators of level three.  Finally, the linear independence of the spanning set is proved by using Kanade--Russell--Kurşung\"{o}z's formula for Capparelli's first identity. 
Moreover, Kanade--Russell--Kurşung\"{o}z's Andrews--Gordon type bivariate generating function (\ref{KursAG}) counts products of $n_1$ quasi-particles of charge 1 and $n_2$ quasi-particles of charge 2 which appear in the quasi-particle basis.

It should be noted that $X\left(\alpha, \beta; \zeta\right)$ appears in \cite[Theorem 13]{Capp2} and plays a key role in the construction of the combinatorial basis for the $A_2^{(2)}$-module $L(3\Lambda_0)$.

\section{Preliminaries on vertex operator construction}\label{section_01}

In this section we introduce the preliminaries on vertex operator construction of the level $1$ standard $A_2^{(2)}$-modules by following \cite{Lep}, \cite{F}. \cite{Capp1} and \cite{N}, by using more or less standard notation.

Let $L=\mathbb{Z}\alpha_1 \oplus \mathbb{Z}\alpha_2$ be a lattice of type $A_2$ with a nonsingular symmetric normalized $\mathbb{Z}$-bilinear form $\left\langle \cdot, \cdot \right\rangle$, with the set of generators
$$
\Phi=\left\lbrace \alpha \in L : \left\langle \alpha, \alpha \right\rangle=2 \right\rbrace,
$$
and with a basis $\Delta=\{\alpha_1,\alpha_2\}$. Following \cite{F} we will use $\nu$ to denote a twisted Coxeter automorphism of order $6$, such that $\nu(\alpha_1)=\alpha_1+\alpha_2$, $\nu(\alpha_2)=-\alpha_1$. This fixed point free isometry of $L$ satisfies
\beqn\label{e2}
\sum_{p \in \mathbb{Z}_6}\nu^{p} =0.
\eeqn

We linearly extend $\left\langle \cdot, \cdot \right\rangle$ and $\nu$ to $\mathfrak{h} = \mathbb{C} \otimes_{\mathbb{Z}}L$. Let $\omega=e^{\frac{2\pi i}{6}}$ denote  the primitive $6$th root of unity. Set
\beqn\label{e3}
\mathfrak{h}=\coprod_{p \in \mathbb{Z}_6} \mathfrak{h}_{(p)},
\eeqn
where $\mathfrak{h}_{(p)}=\left\{ x \in \mathfrak{h}: \nu x=\omega^p x \right\}$. Note that $\mathfrak{h}_{(n)}=\{0\}$ unless $n \equiv \pm 1 (\text{mod} \ 6)$. For $\alpha \in \mathfrak{h}$ and $n \in \mathbb{Z}$ denote by $\alpha_{(n)}$ the projection of $\alpha$ onto $\mathfrak{h}_{(n)}$. Then $\alpha_{(n)}= 0$ unless $n \equiv \pm 1 (\text{mod} \ 6)$.

Consider the $\nu$-twisted affine Lie algebra associated with $\mathfrak{h}$ (viewed as an abelian Lie algebra)
\beqn\label{e4}
\widetilde{\mathfrak{h}}[\nu]=\coprod_{n \in \frac{1}{6}\mathbb{Z}} \mathfrak{h}_{(6n)}\otimes t^{n}\oplus \mathbb{C}c \oplus \mathbb{C}d,
\eeqn
with brackets
\beqn\label{e5}
[\alpha \otimes t^{ \frac{n_1}{6}}, \beta \otimes t^{ \frac{n_2}{6}}] =\frac{n_1}{6} \left\langle \alpha, \beta \right\rangle  \delta_{n_1+n_2,0}c,
\eeqn
\beqn\label{e6}
[d, \alpha \otimes t^{ \frac{n_1}{6}}] = \frac{n_1}{6}\alpha \otimes t^{ \frac{n_1}{6}},
\eeqn
for $\alpha \in \mathfrak{h}_{(n_1)}$,  $\beta \in \mathfrak{h}_{(n_2)}$, $n_1,n_2 \in \mathbb{Z}$, (where we identify $\mathfrak{h}_{(n_1)}=\mathfrak{h}_{(n_1 \ \text{mod} \ 6)}$), and
\beqn\label{e7}
[c, \widetilde{\mathfrak{h}}[\nu]] = 0.
\eeqn
Consider the following subalgebras of $\widetilde{\mathfrak{h}}[\nu]$:
\beqn\label{e8}
\widehat{\mathfrak{h}}[\nu]_{\pm}=\coprod_{\substack{n \in \mathbb{Z},\\ \pm n >0}} \mathfrak{h}_{(n)}\otimes t^{\frac{n}{6}},
\eeqn
\beqn\label{e9}
\mathfrak{b}[\nu]=\widehat{\mathfrak{h}}[\nu]_{+}\oplus \mathbb{C}c \oplus \mathbb{C}d,
\eeqn
and
\beqn\label{e10}
\widehat{\mathfrak{h}}[\nu]=\widehat{\mathfrak{h}}[\nu]_{+}\oplus \widehat{\mathfrak{h}}[\nu]_{-}\oplus \mathbb{C}c,
\eeqn
which is a Heisenberg subalgebra of $\widetilde{\mathfrak{h}}[\nu]$. 

Form the induced $\widetilde{\mathfrak{h}}[\nu]$-module
\beqn\label{e11}
V=U(\widetilde{\mathfrak{h}}[\nu])\otimes_{U(\mathfrak{b}[\nu])}\mathbb{C},
\eeqn
where $\widehat{\mathfrak{h}}[\nu]_{+}\oplus \mathbb{C}d$ acts trivially and $c$ acts as $1$, which is an irreducible $\widehat{\mathfrak{h}}[\nu]$-module  linearly isomorphic to the symmetric algebra $S(\widehat{\mathfrak{h}}[\nu]_{-})$. The action of $d$ defines a $\mathbb{Q}$-grading on $V$
\beqn\label{e12}
V=\coprod_{n \in -\frac{1}{6}\mathbb{Z}}V_{(n)},
\eeqn
where $V_{(n)}$ is the $n$-eigenspace of $d$. 

We will write commuting formal variables $\zeta, \zeta_j$ for formal variables $z^{-1/6}, z_j^{-1/6}$ from \cite{Capp1} and \cite{N}. For $\alpha \in \mathfrak{h}$, $n \in \mathbb{Z}$ we write $\alpha (n)$ for the operator on $V$ associated with $\alpha_{(n)} \otimes t^{\frac{n}{6}} \in \widetilde{\mathfrak{h}}[\nu]$ and set
\beqn\label{e13}
\alpha(\zeta)=\sum_{n \in \mathbb{Z}}\alpha(n)\zeta^n \in \text{End}(V)\{\zeta\}.
\eeqn
For $\alpha \in \mathfrak{h}$ and $p\in \mathbb{Z}_6$ we consider the formal Laurent series
\beqn\label{e14}
E^{\pm}(\nu^p\alpha;\zeta)=\text{exp}\left(6 \sum_{\pm n >0}\alpha(n)\frac{(\omega^p\zeta)^n}{n}\right)=\sum_{\pm n \geq 0}E^{\pm}(\alpha;n)\omega^{pn}\zeta^n ,
\eeqn
and for $\alpha, \beta \in \mathfrak{h}$ the commutation relation
$$
E^{+}(\alpha;\zeta_1)E^{-}(\beta;\zeta_2)=\prod_{p \in \mathbb{Z}_6}\left(1-\omega^{-p}\frac{\zeta_1}{\zeta_2} \right)^{\left\langle \nu^p\alpha, \beta \right\rangle } E^{-}(\beta;\zeta_2)E^{+}(\alpha;\zeta_1).
$$

For $\alpha \in \mathfrak{h}$ let
\beq\label{XEE}
X(\alpha; \zeta)=c_{\alpha}E^{-}(-\alpha;\zeta)E^{+}(-\alpha;\zeta),
\eeq
where $c_{\alpha}=\frac{1+\omega}{36}$, (cf.\cite{KKLW81},\cite{Lep}). For $n \in \mathbb{Z}$ set the component operators $X(\alpha, n)$ by the expansion
\beq\label{e19_label}
X(\alpha; \zeta)=\sum_{n\in \mathbb{Z}}X(\alpha;n)\zeta^{n}.
\eeq
Then we have 
\beq\label{xnuzeta}
X(\nu^p\alpha; n)=X(\alpha;n)\omega^{pn},
\eeq
for all $p \in \mathbb{Z}_6$ and $n \in \mathbb{Z}$, and 
$$
DX(\alpha;\zeta)=-\left[d,X(\alpha;\zeta)\right],
$$
where $D=D_\zeta=\zeta \frac{d}{d\zeta}$.

Let $\alpha, \beta \in \Phi$. Set $I(n) = \left\{ p \in \mathbb{Z}_6:  \left\langle \nu^p\alpha, \beta \right\rangle = n\right\}$, for $n \in \mathbb{Z}$. The commutation relations are given as follows
\begin{align}
\left[ \alpha(\zeta_1), X(\alpha; \zeta_2)\right]  = & \,\sum_{\substack{n \equiv  \pm 1\ \text{mod} \ 6}}\left(\frac{\zeta_1}{\zeta_2}\right)^n\ X(\alpha; \zeta_2),\label{e22}\\
\left[ X(\alpha; \zeta_1), X(\beta;\zeta_2)\right]  = 
&\, \frac{1}{6}\sum_{p \in I(-1)}\epsilon(\nu^p\alpha, \beta)\, 
X(\nu^p\alpha + \beta;\zeta_2)
\,\delta(\omega^{-p}\zeta_1/\zeta_2)\nonumber\\
 & \,+\frac{\epsilon(\beta, -\beta)}{6^2}\sum_{p \in I(-2)}
\left(c\,D\delta(\omega^{-p}\zeta_1/\zeta_2)-6\beta(\zeta_2)\,\zeta_2^{-6}\,\delta(\omega^{-p}\zeta_1/\zeta_2)\right),\label{e220_label}
\end{align}
where $\epsilon:L\times L \rightarrow \mathbb{C}^{\ast}$ is defined by
$$
\epsilon(\alpha, \beta)=(-1)^{\left\langle \nu^{-1}\alpha+\nu^{-2}\alpha,\beta\right\rangle }\omega^{\left\langle \nu^{-1}\alpha+2\nu^{-2}\alpha,\beta\right\rangle }.
$$
We will also  use the following relations
\begin{eqnarray}\label{e24}
	E^{+}(\alpha;\zeta_1) X(\beta; \zeta_2) & =& \prod_{p \in \mathbb{Z}_6}\left(1-\omega^{-p}\frac{\zeta_1}{\zeta_2} \right)^{\left\langle \nu^p\alpha, \beta \right\rangle } X(\beta; \zeta_2)E^{+}(\alpha;\zeta_1), \\
	\label{e25}
	X(\alpha;\zeta_1)E^{-}(\beta;\zeta_2) &=& \prod_{p \in \mathbb{Z}_6}\left(1-\omega^{-p}\frac{\zeta_1}{\zeta_2} \right)^{\left\langle \nu^p\alpha, \beta \right\rangle }E^{-}(\beta;\zeta_2)X(\alpha; \zeta_1).
\end{eqnarray}
In particular, for  $x=\frac{\zeta_1}{\zeta_2}$ we have
\begin{eqnarray*} 
\prod_{p \in \mathbb{Z}_6}\left(1-\omega^{-p}x \right)^{\left\langle \nu^p\alpha, \alpha \right\rangle } &=&\frac{(1 - x)^2(1 -\omega^{-1}x)(1 -\omega^{-5}x)}{(1 -\omega^{-2}x)(1 -\omega^{-3}x)^2(1 -\omega^{-4}x)}, \\
 	\prod_{p \in \mathbb{Z}_6}\left(1-\omega^{-p}x \right)^{\left\langle \nu^p\alpha, \nu \alpha \right\rangle }&=&\frac{(1 - x)(1 -\omega^{-1}x)^2(1 -\omega^{-2}x)}{(1 -\omega^{-3}x)(1 -\omega^{-4}x)^2(1 -\omega^{-5}x)}.
\end{eqnarray*}

Define a Lie algebra $\mathfrak{g}$ as 
\beqn
\mathfrak{g}=\mathfrak{h} \oplus \bigoplus_{\alpha \in \Phi} \mathbb{C}x_{\alpha}
\eeqn
with the brackets
\begin{eqnarray}\nonumber
	[\mathfrak{h},\mathfrak{h}] & =& 0, \\
	\nonumber
	[h, x_{\alpha}] &=& \left\langle h, \alpha \right\rangle x_{\alpha},\\
	\nonumber
	[x_{\alpha}, x_{\beta}]&=&\left\lbrace \begin{array}{ccc} \epsilon(\alpha, -\alpha ) \alpha &  \text{if}  &\alpha + \beta =0\\
		\epsilon(\alpha, \beta ) x_{\alpha + \beta} &  \text{if}  & \left\langle \alpha, \beta \right\rangle=-1\\
		0 &  \text{if} & \left\langle \alpha, \beta \right\rangle\geq 0,\\
	\end{array} \right.
\end{eqnarray} 
for $h \in \mathfrak{h}$ and $\alpha, \beta \in \Phi$. Extend form $\left\langle \cdot, \cdot\right\rangle$ to $\mathfrak{g}$ by
\begin{eqnarray}\nonumber
	\left\langle h, x_{\alpha} \right\rangle &=&  0,\\
	\nonumber
	\left\langle x_{\alpha},  x_{\beta}\right\rangle &=&\left\lbrace \begin{array}{ccc} \epsilon(\alpha, -\alpha ) &  \text{if}  &\alpha + \beta =0\\
		0 &  \text{if}  & \alpha + \beta \neq 0.\\
	\end{array} \right.
\end{eqnarray}
We will use the same symbol $\nu$ for the extension of the automorphism from $\mathfrak{h}$ to $\mathfrak{g}$ defined by
\beqn
\nu x_{\alpha}=x_{\nu \alpha}.
\eeqn
For $n\in\mathbb{Z}$ set the $\omega^n$-eigenspace of $\nu$ in $\g$
\beqn
\g_{(n)}=\{x\in\g\  | \ \nu(x)=\omega^n x\},
\eeqn
and let $\pi_{n}$ be projections of $\g$ onto $\g_{(n)}$. Form the $\nu$-twisted affine Lie algebra $\wtg[\nu]$ associated to $\g$ and $\nu$:
\beqn 
\wtg[\nu]=\coprod_{n\in\frac{1}{6}\mathbb{Z}}\mathfrak{g}_{(6n)}\otimes t^{n} \oplus\mathbb{C}c\oplus\mathbb{C}D.\eeqn
Set $x(n)=x \otimes t^n$ for $x\in\mathfrak{g}_{(n)}$ and $n\in\frac{1}{6}\mathbb{Z}$, then commutation relations in $\wtg[\nu]$ are
\beqn
[x(n_1),y(n_2)]=[x,y](n_1+n_2)+\left<x,y\right>n_1\delta_{n_1+n_2,0}c,
\ [c,\wtg[ \nu]]=0,  \ [D,x(n)]=n x(n).
\eeqn
The Lie algebra $\wtg[\nu]$ is a copy of the principal realization of the affine Lie algebra of type $A_2^{(2)}$, (cf. \cite{F}, \cite{K}, \cite{N}). Let $\left\{h_i, e_i,f_i: i=0,1\right\}$ be canonical generators of $\wtg[\nu]$. We have $c=h_0+2h_1$. Let $\Lambda_i$ ($i = 0,1$) be the fundamental weight such that $\Lambda_i(h_j) = \delta_{ij}$ for all $j=0,1$. Let $L(\Lambda_i)$ be the corresponding standard $A_2^{(2)}$-module of level $\Lambda_i(c)=\Lambda_i(h_0)+2\Lambda_i(h_1)$. By Theorem 9.1 of \cite{Lep} representation of $\widetilde{\mathfrak{h}}[\nu]$ on $V$ extends uniquely to a faithful irreducible Lie algebra representation of $\wtg[\nu]$ on $V$ such that
\beqn
x_{\alpha}(n) \mapsto X(\alpha; n)
\eeqn
for all $n \in \mathbb{Z}$ and $\alpha \in \Phi$.

\section{Quasi-particles for the Lie algebra $A_2^{(2)}$}\label{section_02}

We consider the action of the twisted affine Lie algebra $A_2^{(2)}$ over its standard module $L(\Lambda)$. In  particular, by using the notation from \cite{F}, the action of the formal series \eqref{e19_label} satisfies
\beq\label{xzeta_label}
X(\alpha;\zeta)\in \text{Hom}(L(\Lambda),L(\Lambda)((\zeta))).
\eeq

For fixed $r \in \mathbb{N}$ and $\delta_1, \ldots, \delta_r  \in \Phi$ such that $\left\langle \delta_i, \delta_j\right\rangle \geq 0$ for each $1\leq i,j \leq r$, set 
\beq\label{2e1_povi}
P_{\delta_1, \ldots , \delta_r}\left(\zeta_1, \ldots,\zeta_r\right)=\prod_{1 \leq i < j\leq r} P_{\delta_i, \delta_j}\left(\frac{\zeta_i}{\zeta_j}\right),
\eeq
where
\beq\label{2e2_povi}
P_{\delta_i, \delta_j}\left(\frac{\zeta_i}{\zeta_j}\right)=\prod_{\substack{p \in \mathbb{Z}_m\\ \left\langle \nu^p\delta_i, \delta_j \right\rangle<0}}\left(1-\omega^{-p}\frac{\zeta_i}{\zeta_j}\right)^{-\left\langle \nu^p\delta_i, \delta_j \right\rangle }.
\eeq

\begin{example} For $r=4$ and $\beta=\nu(\alpha)$, we have
\beqn\label{2e30}
P_{\alpha,\beta,\alpha,\beta }\left(\zeta_1,\zeta_2, \zeta_3,\zeta_4\right)=P_{\alpha,\beta}\left(\frac{\zeta_1}{\zeta_2}\right)P_{\alpha,\alpha}\left(\frac{\zeta_1}{\zeta_3}\right)P_{\alpha,\beta}\left(\frac{\zeta_1}{\zeta_4}\right)P_{\beta, \alpha}\left(\frac{\zeta_2}{\zeta_3}\right)P_{\beta, \beta}\left(\frac{\zeta_2}{\zeta_4}\right)P_{\alpha,\beta}\left(\frac{\zeta_3}{\zeta_4}\right),
\eeqn
where
\begin{eqnarray}\nonumber 
P_{\alpha,\beta}\left(\frac{\zeta_i}{\zeta_j}\right)=\overline{P_{\beta, \alpha}}\left(\frac{\zeta_i}{\zeta_j}\right)&=&\left(1+\frac{\zeta_i}{\zeta_j}\right)\left(1-\omega^2\frac{\zeta_i}{\zeta_j}\right)^2\left(1-\omega\frac{\zeta_i}{\zeta_j}\right),\\
\nonumber 
P_{\alpha,\alpha}\left(\frac{\zeta_i}{\zeta_j}\right)=P_{\beta,\beta}\left(\frac{\zeta_i}{\zeta_j}\right)&=&\left(1+\frac{\zeta_i}{\zeta_j}\right)^2\left(1+\omega\frac{\zeta_i}{\zeta_j}\right)\left(1-\omega^2\frac{\zeta_i}{\zeta_j}\right).\end{eqnarray} 
Note that $P_{\alpha,\alpha}(1)\neq 0$ and $P_{\alpha,\beta}(1)\neq 0$.
\end{example}

The commutation relation \eqref{e220_label} implies that for any positive integer $r$ and a permutation $\sigma\in\mathcal{S}_r$ we have
$$
P_{\delta_1, \ldots , \delta_r}\left(\zeta_1, \ldots,\zeta_r\right)X(\delta_1; \zeta_1)  \cdots X(\delta_r; \zeta_r)=
P_{\delta_1, \ldots , \delta_r}\left(\zeta_1, \ldots,\zeta_r\right)X(\delta_{\sigma_1}; \zeta_{\sigma_1})  \cdots X(\delta_{\sigma_r}; \zeta_{\sigma_r}),
$$
so the both sides   of the above equality belong  to
$\text{Hom}(L(\Lambda),L(\Lambda)((\zeta_{\sigma_1}))\ldots((\zeta_{\sigma_r})))$.
As   $\sigma$ was arbitrary, we conclude that 
$P_{\delta_1, \ldots , \delta_r}\left(\zeta_1, \ldots,\zeta_r\right)X(\delta_1; \zeta_1)  \cdots X(\delta_r; \zeta_r)
$ is an element of the intersection
$$
\bigcap_{\sigma\in\mathcal{S}_r}\text{Hom}(L(\Lambda),L(\Lambda)((\zeta_{\sigma_1}))\ldots((\zeta_{\sigma_r})))
= \text{Hom}(L(\Lambda),L(\Lambda)((\zeta_{1},\ldots,\zeta_{ r}))).
$$
Hence,   the limit  
\beq\label{limit}
X\left(\delta_1, \ldots , \delta_r; \zeta\right)=\displaystyle\lim_{\zeta_i \to \zeta}P_{\delta_1, \ldots , \delta_r}\left(\zeta_1, \ldots,\zeta_r\right)X(\delta_1; \zeta_1)\cdots X(\delta_r; \zeta_r)
\eeq
exists. 

By following the approach of S. Ko\v zi\' c and M. Primc in \cite{KP}, we call $X\left(\delta_1, \ldots , \delta_r;n\right)$, a coefficient in
\beq\label{Xdelta_label}
X\left(\delta_1, \ldots , \delta_r;\zeta\right)=\displaystyle\sum_{n \in \mathbb{Z}}X\left(\delta_1, \ldots , \delta_r;n\right)\zeta^{n},
\eeq
the {\em quasi-particle of charge $r$ and energy $n$} (see also  \cite{FS}, \cite{G}). From   \cite[Proposition 3.3]{Capp1} (see also \cite{MP}), follows that the quasi-particle of charge $r$ and energy $n$ equals the infinite sum $\sum_{n_1+\cdots+n_r=n}f(n_1, \ldots, n_r)$ of the summable family 
$$
\left\{f(n_1, \ldots, n_r) : n_1+ \ldots + n_r = n\right\}
$$
which is defined by
$$
\displaystyle\sum_{n_1, \ldots, n_r \in \mathbb{Z}}f(n_1, \ldots, n_r)\zeta_1^{n_1} \ldots\zeta_r^{n_r}=
P_{\delta_1, \ldots , \delta_r}\left(\zeta_1, \ldots,\zeta_r\right)X(\delta_1; \zeta_1)\cdots X(\delta_r; \zeta_r).
$$

We will use the following Lemma, which follows from   relations \eqref{e22}, \eqref{e24} and \eqref{e25}. 
\begin{lem}\label{implem}
 For fixed $\alpha, \delta \in \Phi$ and $\beta=\nu (\alpha)$, we have
\begin{itemize}
	\item[a)] $[\alpha(\zeta_1), X(\alpha, \beta; \zeta_2)]=\sum_{\substack{n \equiv  \pm 1\ \text{mod} \ 6}}\left(1+\omega^{-n}\right)\left(\frac{\zeta_1}{\zeta_2}\right)^n \ X(\alpha,\beta; \zeta_2)$,
	\item[b)] $X(\alpha, \beta; \zeta_1)E^{-}(\delta; \zeta_2)=\prod_{p \in \mathbb{Z}_6}\left(1-\omega^{-p}\frac{\zeta_1}{\zeta_2} \right)^{\left\langle \nu^p(\alpha+\beta), \delta \right\rangle }E^{-}(\delta; \zeta_2)X(\alpha, \beta; \zeta_1)$,
	\item[c)] $E^{+}(\delta; \zeta_1)X(\alpha, \beta; \zeta_2)=\prod_{p \in \mathbb{Z}_6}\left(1-\omega^{-p}\frac{\zeta_1}{\zeta_2} \right)^{\left\langle \nu^p\delta, \alpha+\beta \right\rangle }X(\alpha, \beta; \zeta_2)E^{+}(\delta; \zeta_1)$.
	\end{itemize}
\end{lem}

\section{Quasi-particle basis of the $A_2^{(2)}$-module $L(3\Lambda_0)$}\label{section_03}

As with the previous section, we consider the action of the twisted affine Lie algebra $A_2^{(2)}$ over its standard module $L(3\Lambda_0)$. 
Our goal is to present our main result, Theorem \ref{main_thm} below, and introduce some related terminology which we use in its proof, which occupies Sections \ref{section_04} and \ref{section_05}.  
We denote by $v_{3\Lambda_0}$ be the highest weight vector of $L(3\Lambda_0)$. 

Recall the following $\widehat{\mathfrak{h}}[\nu]$-filtration on $L(3\Lambda_0)$ (see \cite{LW84}): 
\beq \label{50} 
\{0\}=L(3\Lambda_0)_{(-1)}\subset L(3\Lambda_0)_{(0)}\subset L(3\Lambda_0)_{(1)}\subset \ldots \subset L(3\Lambda_0),
\eeq
\beqn
L(3\Lambda_0)=\bigcup_{j \geq 0}L(3\Lambda_0)_{(j)},
\eeqn
where for $j\in \mathbb{Z}$
\beq \label{51}
L(3\Lambda_0)_{(j)}=U_{(j)}v_{3\Lambda_0} \subset L(3\Lambda_0),  
\eeq
$U=U(\widetilde{\mathfrak{g}}[\nu])$,
$U_{(j)}=\{0\}$ for $j<0$, $U_{(0)}=U(\widehat{\mathfrak{h}}[\nu])$ and for $j>0$
\beq \label{52}
U_{(j)}= \text{span}\left\{ x_1 \cdots x_n \in U:x_i \in \widetilde{\mathfrak{g}}[\nu] \text{ and at most} \ j \ \text{of} \ x_i \ \text{lie outside} \ \widehat{\mathfrak{h}}[\nu] \right\}.
\eeq

\begin{defn}
Denote the operator $X(\alpha  ;\zeta)$,   given by  \eqref{e19_label} for $\alpha=\alpha_1$,  by
$$
X^{(1)}(\zeta)= \sum_{n\in\mathbb{Z}} X^{(1)}(n)\,\zeta^n,
$$
and  $X(\alpha,\beta;\zeta)$, as defined by \eqref{limit} for $r=2,\, \delta_1=\alpha=\alpha_1,\,\delta_2 =\beta=\alpha_1+\alpha_2$, by
$$
X^{(2)}(\zeta)= \sum_{n\in\mathbb{Z}} X^{(2)}(n)\,\zeta^n.
$$
Following the terminology from Section \ref{section_02}, we refer to the coefficients $X^{(1)}(n)$ (resp. $X^{(2)}(n)$)   as quasi-particles of charge 1 (resp. charge 2) and energy $n$.
\end{defn}

By using (\ref{e22}) and (\ref{e220_label}), for any permutation $\sigma\in\mathcal{S}_n$ and $m_1, \ldots ,m_n \in \mathbb{Z}$, we have
\beq \label{53}
X^{(1)}(m_{\sigma(1)})  \cdots X^{(1)}(m_{\sigma(n)})v_{3\Lambda_0}- X^{(1)}(m_{1}) \cdots X^{(1)}(m_{n})v_{3\Lambda_0} \in L(3\Lambda_0)_{(n-1)}.
\eeq
 From (\ref{53}) follows that for any permutation $\sigma\in\mathcal{S}_n$, we have
\beq \label{54}
X^{(2)}(m_{\sigma(1)})  \cdots X^{(2)}(m_{\sigma(n)})v_{3\Lambda_0}- X^{(2)}(m_{1}) \cdots X^{(2)}(m_{n})v_{3\Lambda_0} \in L(3\Lambda_0)_{(2n-1)}.
\eeq

By the Poincar\'{e}--Birkoff--Witt theorem for the universal enveloping algebra $U=U(\widetilde{\mathfrak{g}}[\nu])$, the spanning set of $L(3\Lambda_0)$ can be described as the set of monomial vectors
	$$
	\left\{\alpha(i_1)\cdots \alpha(i_r)\, 
	X^{(1)}(j_1)\cdots X^{(1)}(j_v)\, v_{3\Lambda_0} \right\}
	$$
such that $i_1\leq i_2\leq \ldots \leq i_r< 0$ and $i_p \equiv \pm 1\mod 6$ for each $1 \leq p \leq r$ and such that $j_1\leq j_2\leq \ldots \leq j_v< 0$. 

As in \cite{KP} (see also 
\cite{G}), for a quasi-particle monomial $b$ of the form
\beq \label{qpm}
b= X^{(1)}(j_1)\cdots X^{(1)}(j_s)\,
X^{(2)}(k_1)\cdots X^{(2)}(k_t),
\eeq
we define its {\em color-type} as $(s,t)$; {\em (total) charge} as $s+2t$; {\em degree-type} as $(j_1,\cdots ,j_s;\,
k_1,\cdots ,k_t)$ and {\em (total) degree} as $j_1 + \cdots +j_s+ 
k_1+\cdots +k_t$.

On the set of quasi-particle monomials \eqref{qpm} we define a linear order "$\leq$'' as follows. For quasi-particle monomials $b,b'$ as in \eqref{qpm}, we will write 
$b\leq b'$
if one of the following conditions holds:
\begin{itemize}
	\item  total charge of $b$ is greater than total charge of $b'$;
	\item total charges of $b$ and $b'$ are the same and color-type of $b$ is less than color-type of $b'$ with respect to the reverse lexicographic order;
	\item   color-types of $b$ and $b'$ are the same and degree-type of $b$ is less than degree-type of $b'$ with respect to the reverse lexicographic order.
	\end{itemize}
	Finally, let us extend the linear order to the set of all monomials
	\beq\label{monnomials}
	a\,b=\underbrace{
	\alpha(i_1)\cdots \alpha(i_r)}_{a}\, 
\underbrace{X^{(1)}(j_1)\cdots X^{(1)}(j_s)\,
X^{(2)}(k_1)\cdots X^{(2)}(k_t)}_{b}.
\eeq
	First, define the {\em degree-type} of $a=\alpha(i_1)\cdots \alpha(i_r)$ as
	$(i_1,\ldots ,i_r)$. Then, for monomials $ab$ and $a'b'$ as in \eqref{monnomials},  we write
	\beq\label{lin_ord}
	a\,b\leq a'\, b'
	\eeq
	if $b\leq b'$ or $b=b'$ and degree-type of $a$ is less than degree-type of $a'$ with respect to the reverse lexicographic order.
Also, we generalize the notion of (total) charge to all monomials $ab$ of the form \eqref{monnomials} by defining it to be the (total) charge of $b$, which is $s+2t$. Hence, the generators of the Heisenberg subalgebra do not affect the total charge. 

The next theorem is the main result of this paper.
\begin{thm}\label{main_thm}
The monomial vectors
\beq\label{monomial_vectors}
\alpha(i_1)\cdots \alpha(i_r)\, 
X^{(1)}(j_1)\cdots X^{(1)}(j_s)\,
X^{(2)}(k_1)\cdots X^{(2)}(k_t)\, v_{3\Lambda_0}
\eeq
such that
\begin{align}
&i_1\leq i_2\leq \ldots \leq i_r< 0,\quad i_p \equiv \pm 1\mod 6,\label{1_cond}\\
&j_1\leq j_2\leq \ldots \leq j_s< 0,\label{2_cond}\\
&k_1\leq k_2\leq \ldots \leq k_t< 0,\quad k_p \equiv 0\mod 3,\label{3_cond}\\
& j_p\leq j_{p+1}-4,\quad k_p\leq k_{p+1}-12,\label{diff_cond}\\
& j_s\leq -2-6t,\quad k_t\leq -6\label{init_cond}
\end{align}
form a basis of the standard module $L(3\Lambda_0)$.
\end{thm}

Note that by \eqref{diff_cond},  the degrees of two adjacent quasi-particles of charge one (resp. charge two) need to differ by at least $4$ (resp. $12$). We shall often refer to \eqref{diff_cond} as {\em difference conditions}. As for \eqref{init_cond}, it implies that the degrees of all quasi-particles of charge one (resp. charge two) are less than or equal to $-2$ (resp. $-6$). We refer to \eqref{init_cond} as {\em initial conditions}.

\section{Relations for quasi-particles}\label{section_04}

In this section, we study relations satisfied by the operators $X^{(1)}(\zeta)$ and  $X^{(2)}(\zeta)$ on $L(3\Lambda_0)$ and obtain some intermediate results  which  we use in the proof of Theorem \ref{main_thm}.

For the rest of the paper set
\beq\label{korijeni}
\alpha=\alpha_1,\quad \beta=\alpha_1+\alpha_2=\nu(\alpha)\quad\text{and}\quad \gamma=\alpha_2=\nu^2(\alpha).
\eeq

By the special case of \eqref{limit},  we have
\beq\label{XAA}
X(\alpha,\alpha;\zeta)
=
\lim_{\zeta_i\to \zeta} P_{\alpha,\alpha}(\zeta_1,\zeta_2) X^{(1)}(\zeta_1)X^{(1)}(\zeta_2).
\eeq
Combining the explicit formula \eqref{XEE} and the embedding 
\beq\label{embedding}
L(3\Lambda_0) \hookrightarrow L(\Lambda_0) \otimes L(\Lambda_0) \otimes L(\Lambda_0) ,
\eeq 
the expression in \eqref{XAA} can be written as
$$
\lim_{\zeta_i\to \zeta} P_{\alpha,\alpha}(\zeta_1,\zeta_2) X^{(1)}(\zeta_1)X^{(1)}(\zeta_2)
=\kappa\,
E^-(-\alpha;\zeta)\, X(-\alpha;\zeta)\,E^+(-\alpha;\zeta)
$$
for some nonzero  $\kappa\in\mathbb{C}$ (cf. \cite[Theorem 6.6]{MP}). As $ \nu^3(\alpha)=-\alpha$, we can use \eqref{xnuzeta} to write the above equality as
\beq\label{XAA2}
\lim_{\zeta_i\to \zeta} P_{\alpha,\alpha}(\zeta_1,\zeta_2) X^{(1)}(\zeta_1)X^{(1)}(\zeta_2)
=\kappa\,
E^-(-\alpha;\zeta)\, X^{(1)}( -\zeta)\,E^+(-\alpha;\zeta).
\eeq

Next, consider the following particular case of \eqref{limit}:
$$
\lim_{\zeta_i\to \zeta} P_{\alpha,\beta}(\zeta_1,\zeta_2) X^{(1)}(\zeta_1)X (\beta;\zeta_2)
=X^{(2)}(\zeta).
$$
As $\nu(\alpha)=\beta$, by employing \eqref{xnuzeta} once again, one can write the above identity  as
\beq\label{XAB}
\lim_{\zeta_i\to \zeta} P_{\alpha,\beta}(\zeta_1,\zeta_2) X^{(1)}(\zeta_1)X^{(1)} (\omega\zeta_2)
=X^{(2)}(\zeta).
\eeq 

\begin{lem}\label{lemma_41}
For any   $i,j\in\mathbb{Z}$   and   $v\in L(3\Lambda_0)$, the   vector
\beq\label{vec1}
X^{(1)}(i)X^{(1)}(j)v\quad\text{such that}\quad\left|i-j\right|<4
\eeq
  can be expressed in terms of the   vectors 
\begin{align}
& \alpha(i+j)\,v, \  X^{(1)}(i+j)\,v,  \ X^{(2)}(i+j)\,v,\label{vec2}\\
&X^{(1)}(a) X^{(1)}(b) v  \text{ such that } \left|a-b\right|\geq 4\text{ and }a+b =i+j,\label{vec3}\\\
&E^-(-\alpha;-a)X^{(1)}(b)v \text{ such that }  -a+b=i+j \text{ and } a>0,\label{vec4}\\
&X^{(1)}(b)E^+(-\alpha;a)v   \text{ such that }    a+b=i+j \text{ and } a>0,\label{vec5}\\
&E^-(-\alpha;-a_1)X^{(1)}(b)E^+(-\alpha;a_2)v   \text{ such that }     a_2-a_1+b=i+j \text{ and } a_1,a_2>0.\label{vec6}
\end{align}
\end{lem}

\begin{proof}
The lemma follows from the quasi-particle relations \eqref{XAA2} and  \eqref{XAB}. More specifically, by extracting the coefficients of $\zeta^{i+j}$ on their left-hand sides (resp. right-hand sides), one obtains linear combinations of the vectors \eqref{vec1} and \eqref{vec3} (resp. \eqref{vec2} and \eqref{vec4}--\eqref{vec6}). Hence, moving the vectors \eqref{vec3} to the right-hand sides, one expresses two linear combinations of the vectors \eqref{vec1} in terms of \eqref{vec2}--\eqref{vec6}. Due to \eqref{e220_label}, this can be regarded as a system of two linear equations in two indeterminates,
$$
X^{(1)}(i)X^{(1)}(j)v\quad\text{such that}\quad\left|i-j\right|<4\quad\text{and}\quad i\leq j.
$$
 By a straightforward computation, one checks that the matrix of the system  is regular, which implies the lemma.
\end{proof}

Recall that the roots $\alpha,\,\beta$ and $\gamma$ are given by \eqref{korijeni}.
By using   \eqref{XEE} and the embedding \eqref{embedding}, we find that
$$
E^-(\alpha;\zeta)X(\alpha,\beta;\zeta)
=  X(-\alpha,\gamma;\zeta)E^+(-\alpha;\zeta)   
$$
on $L(3\Lambda_0)$ (cf. \cite[Theorem 13]{Capp2}).
Due to $\nu^2(\beta)=-\alpha$, 
 $\nu^2(\alpha)=\gamma$ and \eqref{xnuzeta}, this yields 
\beq\label{Rrel}
R(\zeta)\coloneqq
\sum_{n\in\mathbb{Z}} R(n)\,\zeta^n\coloneqq
E^-(\alpha;\zeta)X^{(2)}( \zeta)
-X^{(2)}(\omega^2\zeta)E^+(-\alpha;\zeta)=0.
\eeq

\begin{lem}\label{lemma_42}
For any   $n\in\mathbb{Z}$ such that $n\not\equiv 0\mod 3$ and   $v\in L(3\Lambda_0)$, the   vector $X^{(2)}(n)v$ can be expressed in terms of the vectors
$$
E^-(\alpha;-i)X^{(2)}(n+i)v \quad\text{and}\quad X^{(2)}(n-i) E^+(\alpha;i)v \quad\text{with}\quad i>0.
$$
\end{lem}

\begin{proof}
The lemma is an immediate consequence of \eqref{Rrel}. Indeed, by applying this relation on a vector $v\in L(3\Lambda_0)$ and then extracting the coefficients of   $\zeta^n$, we get
\begin{align*}
\left(1-\omega^{2n}\right)X^{(2)}(n)v+\sum_{i>0} E^-(\alpha;-i)X^{(2)}(n+i)v
-\sum_{i>0} \omega^{2(n-i)}X^{(2)}(n-i) E^+(\alpha;i)v=0
\end{align*}
since $R(n)v=0$ for all $n$. 
Dividing the equality by $1-\omega^{2n}$, which is nonzero for $n\not\equiv 0\mod 3$, we obtain the required decomposition of $X^{(2)}(n)v$.
\end{proof}

By the commutation relation \eqref{e220_label}, we have
$$
P_{\alpha,\beta}(\zeta_1,\zeta_2)\left[ X(\alpha ; \zeta_1), X(\beta;\zeta_2)\right]  = 0 \quad\text{for any }\alpha,\beta\in\Phi,
$$
as the polynomial $P_{\alpha,\beta}(\zeta_1,\zeta_2)$ annihilates the delta functions which appear on the right-hand side of \eqref{e220_label}. Let $n $ and $m$ be   nonnegative integers. Analogously, applying   $D_{\zeta_1}^n D_{\zeta_2}^m$ to \eqref{e220_label} and then multiplying the resulting identity by $P_{\alpha,\beta}(\zeta_1,\zeta_2)^{n+m+1}$, we obtain
\beq\label{rl43}
P_{\alpha,\beta}(\zeta_1,\zeta_2)^{n+m+1}\left[D_{\zeta_1}^n X(\alpha ; \zeta_1),D_{\zeta_2}^m X(\beta;\zeta_2)\right]  = 0 \quad\text{for any }\alpha,\beta\in\Phi, \ n,m\geq 0.
\eeq

\begin{lem}
Fix   positive integers $r$ and $s$ and roots $\gamma_i,\delta_j\in\left\{\alpha,\beta\right\}$, where $i=1,\ldots ,r$ and $j=1,\ldots ,s$. For any nonnegative integers $n_1,\ldots ,n_r$, there exists a polynomial 
$P(\zeta_1,\ldots ,\zeta_r,\xi_1,\ldots,\xi_s)$ such that $P(1,\ldots ,1)\neq 0$ and such that there exists a limit
\beq\label{limex}
\lim_{\zeta_i,\xi_j\to\chi}
P(\zeta_1,\ldots ,\zeta_r,\xi_1,\ldots,\xi_s)
\left(\prod_{i=1}^r  D_{\zeta_i}^{n_i} X(\gamma_i;\zeta_i)\right)
\left(\prod_{k=1}^s    X(\delta_k;\xi_k)\right).
\eeq
\end{lem}

\begin{proof}
It is evident from \eqref{rl43} that the limit in \eqref{limex} exists for 
\begin{align*}
&P(\zeta_1,\ldots ,\zeta_r,\xi_1,\ldots,\xi_s)\\
=&\prod_{1\leq i<j\leq r} P_{\gamma_i,\gamma_j}(\zeta_i,\zeta_j)^{n_i +n_j +1}
\prod_{i=1}^r\prod_{k=1}^s P_{\gamma_i,\delta_k}(\zeta_i,\xi_k)^{n_i  +1}
 \prod_{1\leq k<l\leq s}  P_{\delta_k,\delta_l}(\xi_k,\xi_l) .
\end{align*}
Finally, the assertion $P(1,\ldots ,1)\neq 0$   follows directly from \eqref{2e1_povi}, \eqref{2e2_povi} and the lemma assumption that all
$\gamma_i$ and $\delta_j$ belong to $\left\{\alpha,\beta\right\}$.
\end{proof}

By the commutation relation \eqref{e220_label},   there exists a positive integer $N$ such that for all  $i_1,i_2,j_1,j_2\geq 0$ satisfying $   i_1+i_2+ j_1+j_2 \leq 3$, the limits
\beq\label{XABAB3}
\lim_{\zeta_1,\zeta_2\to\zeta}
 D_{\zeta_1}^{i_1} D_{\zeta_2}^{j_1} \left(P_{\alpha,\beta,\alpha,\beta}(\zeta_1,\zeta_1,\zeta_2,\zeta_2)^N\right)
D_{\zeta_1}^{i_2}\left(X^{(2)}(\zeta_1)\right) 
D_{\zeta_2}^{j_2}\left(X^{(2)}(\zeta_2)\right)
\eeq
exist. 
Note that 
\beq\label{XABAB}
X(\alpha,\beta,\alpha,\beta;\zeta)
=\kappa\lim_{\zeta_1,\zeta_2\to\zeta}
P_{\alpha,\beta,\alpha,\beta}(\zeta_1,\zeta_1,\zeta_2,\zeta_2)^N
X^{(2)}(\zeta_1) X^{(2)}(\zeta_2)
\eeq
for some nonzero $\kappa\in\mathbb{C}$. 
By applying $D_\zeta^m$ with $m=0,1,2,3$ to \eqref{XABAB}, we obtain four quasi-particle relations, 
\begin{align}
&D_\zeta^m X(\alpha,\beta,\alpha,\beta;\zeta)\nonumber\\
= &\, \kappa\lim_{\zeta_1,\zeta_2\to\zeta}
\left(D_{\zeta_1}+D_{\zeta_2}\right)^m \left(P_{\alpha,\beta,\alpha,\beta}(\zeta_1,\zeta_1,\zeta_2,\zeta_2)^N
X^{(2)}(\zeta_1) X^{(2)}(\zeta_2)\right),\,\,\label{XABAB2}
m=0,1,2,3.
\end{align}

 	In the next lemmas, we will need the notion of (total) charge introduced in Section \ref{section_03}.
More specifically, for any $v\in L(3\Lambda_0)$, we shall say that the (total) charge of $abv$ is {\em (strictly) smaller} than the total charge of $a' b' v$, where  $ab$ and $a'b'$ are as in \eqref{monnomials}, if the total charge of $ab $ is (strictly)  smaller  than the (total) charge of $a' b' $.

\begin{lem}\label{lemma_44}
For any   $n,j\in\mathbb{Z}$   and   $v\in L(3\Lambda_0)$, the   vectors
\begin{gather*}
X^{(2)}(j)X^{(2)}(n-j)v,\,\,
X^{(2)}(j+1)X^{(2)}(n-j-1)v,\\
X^{(2)}(j+2)X^{(2)}(n-j-2)v,\,\,
X^{(2)}(j+3)X^{(2)}(n-j-3)v 
\end{gather*}
can be expressed in terms of the vectors
$$
X^{(2)}(i)X^{(2)}(n-i)v\quad\text{with }
i\in\mathbb{Z}\setminus\left\{j,j+1,j+2,j+3\right\}
$$
and the vectors of   strictly smaller total charge. 
\end{lem}

\begin{proof}
The lemma can be verified by an argument which closely follows the proof of \cite[Lemma 4.5]{KP} and relies on the quasi-particle relations \eqref{XABAB2}.
In particular, it makes use of the observation that the right-hand sides of \eqref{XABAB2} can be written as linear combinations of all elements \eqref{XABAB3} such that $   i_1+i_2+j_3+j_4=m$.
\end{proof}

\begin{lem}\label{lemma_45}
Let  $i$ be an integer  and   $v\in L(3\Lambda_0)$.
\begin{enumerate}[a)]
\item The  vectors
$$
X^{(2)}(3i-3)X^{(2)}(3i+3)v,\,\,X^{(2)}(3i )X^{(2)}(3i )v,\,\,X^{(2)}(3i+3)X^{(2)}(3i-3)v
$$
can be expressed in terms of the vectors 
$$
X^{(2)}(s)X^{(2)}(6i-s)v\quad\text{with }
s\in\mathbb{Z}\setminus\left\{3i-3,3i,3i+3\right\}
$$
and   the vectors of   strictly smaller total charge. 
\item The  vectors
\begin{gather*}
X^{(2)}(3i-6)X^{(2)}(3i+3)v,\,\,X^{(2)}(3i-3)X^{(2)}(3i )v,\\
X^{(2)}(3i )X^{(2)}(3i-3)v,\,\,X^{(2)}(3i+3)X^{(2)}(3i-6)v 
\end{gather*}
can be expressed in terms of the vectors 
$$
X^{(2)}(s)X^{(2)}(6i-3-s)v\quad\text{with }
s\in\mathbb{Z}\setminus\left\{3i-6,3i-3,3i,3i+3\right\}
$$
and   the vectors of   strictly smaller total charge. 
\end{enumerate}
\end{lem}

\begin{proof}
 Both statements a) and b) in the lemma can be proven by combining Lemma \ref{lemma_42} and Lemma \ref{lemma_44}. In the case of the statement a), note that in a sequence of 12 monomial vectors,
$$X^{(2)}(3i-6)X^{(2)}(3i+6)v,\,X^{(2)}(3i-5)X^{(2)}(3i+5)v,\,\ldots$$
$$ \ \ \ \ \ \ \ \ \ \ \ \ \ \ \ \ \ \ \ \ \ \ \ \ \ \ \ \ \ldots \,X^{(2)}(3i )X^{(2)}(3i )v,\,\ldots\,\,X^{(2)}(3i+5)X^{(2)}(3i-5)v,$$
by using Lemma \ref{lemma_42} all monomial vectors of the form 
$$bv=X^{(2)}(3i-j)X^{(2)}(3i+j)v, \ \ j \equiv \pm 1 \ \text{mod} \ 3$$
can be expressed as a linear combination of monomial vectors of the form $b'v$, with $b'$ of higher total degree type then $b$. Now, on the set of monomial vectors $\{ X^{(2)}(3i-3)X^{(2)}(3i+3)v,\,\,X^{(2)}(3i )X^{(2)}(3i )v,\,\,X^{(2)}(3i+3)X^{(2)}(3i-3)v\}$ we apply  Lemma \ref{lemma_44}.

The proof of b) goes in an analogous way.
\end{proof}

Recall that, due to \eqref{xnuzeta}, we have $X (\beta;\zeta)=X^{(1)}(\omega\zeta )$.
Consider the identities
\begin{align}
& X(\alpha,\alpha,\beta;\zeta)
=\kappa'_1\lim_{\zeta_i\to\zeta}
P_{\alpha,\alpha,\beta}(\zeta_1,\zeta_2,\zeta_2)
X^{(1)}(\zeta_1)X^{(2)}(\zeta_2),\label{rel12_a}\\
& X(\beta,\alpha,\beta;\zeta)
=\kappa'_2 \lim_{\zeta_i\to\zeta}
P_{\beta,\alpha,\beta}(\zeta_1,\zeta_2,\zeta_2)
X^{(1)}(\omega\zeta_1)X^{(2)}(\zeta_2),\label{rel12_b}
\end{align}
which hold for some nonzero $\kappa'_1 ,\kappa'_2\in\mathbb{C}$.
On the other hand, by using the explicit formula for the level one realization \eqref{XEE} and the embedding \eqref{embedding}, we find
\begin{align}
&X(\alpha,\alpha,\beta;\zeta)
=\kappa''_1\,E^-(-\alpha;\zeta)
X(\gamma;\zeta)
E^+(-\alpha;\zeta),\label{rel12_c}\\
&X(\beta,\alpha,\beta;\zeta)
=\kappa''_2\,E^-(-\beta;\zeta)
X(-\gamma;\zeta)
E^+(-\beta;\zeta) \label{rel12_d}
\end{align}
for some nonzero $\kappa''_1 ,\kappa''_2\in\mathbb{C}$. Finally, by combining the equalities \eqref{rel12_a}--\eqref{rel12_d} and
$$
X(\gamma;\zeta)=X^{(1)}(\omega^2\zeta)\quad\text{and}\quad
X(-\gamma;\zeta)=X^{(1)}(\omega^5\zeta) 
$$
(which follow from \eqref{xnuzeta}), we obtain the relations
\begin{align}
&\lim_{\zeta_i\to\zeta}
P_{\alpha,\alpha,\beta}(\zeta_1,\zeta_2,\zeta_2)
X^{(1)}(\zeta_1)X^{(2)}(\zeta_2)
=\kappa_1\,E^-(-\alpha;\zeta)
X^{(1)}(\omega^2\zeta)
E^+(-\alpha;\zeta)
,\label{rel12_e}\\
&
\lim_{\zeta_i\to\zeta}
P_{\beta,\alpha,\beta}(\zeta_1,\zeta_2,\zeta_2)
X^{(1)}(\omega\zeta_1)X^{(2)}(\zeta_2)
=\kappa_2\,E^-(-\beta;\zeta)
X^{(1)}(\omega^5\zeta)
E^+(-\beta;\zeta)
 \label{rel12_f}
\end{align}
for some nonzero $\kappa_1 ,\kappa_2\in\mathbb{C}$.

\begin{lem}\label{lemma_46}
For any $i,j\in\mathbb{Z}$ and $v\in L(3\Lambda_0)$, the vectors
\beq\label{lem_46_1}
X^{(1)}(i)X^{(2)}(3j-3)v\quad\text{and}\quad X^{(1)}(i-3)X^{(2)}(3j)v
\eeq
can be expressed in terms of the vectors 
\beq\label{lem_46_2}
X^{(1)}(r)X^{(2)}(i+3j-3-r)v\quad\text{with }r\in \mathbb{Z}\setminus\left\{ i-3,i \right\}
\eeq
and   the vectors of   strictly smaller total charge. 
\end{lem}

\begin{proof}
The lemma can be verified by  adjusting the arguments from the proof of Lemma \ref{lemma_41}. More specifically, by using the quasi-particle relations \eqref{rel12_e} and \eqref{rel12_f} and the commutation relation \eqref{e220_label}, one   expresses two linear combinations of vectors \eqref{lem_46_1} in terms of vectors
\eqref{lem_46_2} and the vectors of   strictly smaller length. Finally, one checks directly that the $2\times 2$ matrix, such that its entries are coefficients of these linear combinations, is regular, which implies the lemma.
\end{proof}

\section{Proof of Theorem \ref{main_thm}}\label{section_05}
In this section, we  prove the main theorem. 
Clearly, the set of all   vectors as in \eqref{monomial_vectors},
$$
\alpha(i_1)\cdots \alpha(i_r)\, 
X^{(1)}(j_1)\cdots X^{(1)}(j_s)\,
X^{(2)}(k_1)\cdots X^{(2)}(k_t)\, v_{3\Lambda_0},
$$
such that
$i_1,\ldots,i_r, j_1,\ldots ,j_s,   k_1,\ldots , k_t<0$, i.e., which do not necessarily satisfy the conditions \eqref{1_cond}--\eqref{init_cond},
spans the standard module $L(3\Lambda_0)$.
In Subsection \ref{section_0501}, we show that it is sufficient to consider the
monomial vectors which satisfy initial conditions \eqref{init_cond}.
Next, in Subsection  \ref{section_0502}, we refine this   by proving that the 
 vectors  \eqref{monomial_vectors} which satisfy all conditions \eqref{1_cond}--\eqref{init_cond} span $L(3\Lambda_0)$. Finally, in Subsection  \ref{section_0503},
 we use the Kanade--Russell--Kurşung\"{o}z formula to establish the linear independence of these vectors.

\subsection{Initial conditions}\label{section_0501}

Using \eqref{XEE} we get 
\beq \label{ICX1}
X^{(1)}(-1)v_{3\Lambda_0} \in L(3\Lambda_0)_{(0)}.
\eeq
\begin{rem}\label{remi1}
By using this relation as well as commutation relations  \eqref{e24}  and  \eqref{e25}, every monomial vector of the form $\cdots X^{(1)}(m)X^{(1)}(-1)v_{3\Lambda_0}$ can be expressed as a linear combination of   monomial vectors of the form $\cdots X^{(1)}(m')v_{3\Lambda_0}$, where $m' >m$. 
Finally, the monomial vectors of the form $\cdots X^{(1)}(m')v_{3\Lambda_0}$ have smaller total charge than  $\cdots X^{(1)}(m)X^{(1)}(-1)v_{3\Lambda_0}$.  
\end{rem}
Using (\ref{53}) we get 
\beqn
X^{(2)}(-3)v_{3\Lambda_0}\equiv \sum_{\substack{m_1+m_2=-3,\\ m_1, m_2 \leq -1}}\kappa_{m_1,m_2}X^{(1)}(m_{1})X^{(1)}(m_{2})v_{3\Lambda_0} \ \ \ \text{mod} \ L(3\Lambda_0)_{(1)},
\eeqn
where $\kappa_{m_1,m_2}\in \mathbb{C}\setminus \{0\}$. Now, relation  \eqref{ICX1}  implies  
\beq \label{ICX2_0}
X^{(2)}(-3)v_{3\Lambda_0} \in L(3\Lambda_0)_{(1)}.
\eeq
\begin{rem}\label{remi2} Relation  \eqref{ICX2_0}, together with Lemma \ref{lemma_42} and commutation relations in Lemma \ref{implem}, implies that every monomial vector of the form $\cdots X^{(2)}(m)v_{3\Lambda_0}$, with $m>-6$ can be expressed as a linear combination of   monomial vectors which 
possess   smaller total charge, so that they
are  greater   with respect to the linear order   \eqref{lin_ord}. 
\end{rem}

\subsection{Spanning set}\label{section_0502}

 We will show that the monomial vectors   \eqref{monomial_vectors},
\beq\label{span1}
\alpha(i_1)\cdots \alpha(i_r)\, 
X^{(1)}(j_1)\cdots X^{(1)}(j_s)\,
X^{(2)}(k_1)\cdots X^{(2)}(k_t)\, v_{3\Lambda_0},
\eeq
 which satisfy all conditions \eqref{1_cond}--\eqref{init_cond} span $L(3\Lambda_0)$.
The proof goes by induction
  on the total degree $D=I+J+K$, where $I=|i_1|+\cdots +|i_r|$, $J=|j_1|+\cdots +|j_s|$ and $K= |k_1|+\cdots +|k_t|$, and by an induction on the linear order \eqref{lin_ord}. 
	By using commutation relations \eqref{e22} and \eqref{e220_label}, we can write any monomial vector $bv_{3\Lambda_0}$ as a linear combination of monomial vectors of the form as in \eqref{span1} which satisfy \eqref{1_cond}, \eqref{2_cond} and
	 $k_1\leq k_2\leq \ldots \leq k_t< 0$. By Remarks \ref{remi1} and \ref{remi2}, we can further assume that 
	$$
	j_p\leq -2\text{ for every }1\leq p\leq s \qquad\text{and}\qquad  k_p\leq -6
	\text{ for every } 1\leq p\leq t.
	$$

The condition $k_p \equiv 0 \ \text{mod} \ 3$ for all $1 \leq p \leq t$ from \eqref{3_cond}  appears as a consequence of Lemma \ref{lemma_42}. By combining it with Lemma \ref{implem}, the monomial vector $bv_{3\Lambda_0}$ with a factor $X^{(2)}(k_p)$, such that $k_p \not \equiv 0 \ \text{mod} \ 3$, can be replaced with linear combination of monomial vectors $b'v_{3\Lambda_0}$ of the equal total degree and equal color-type as $b$, but $b'$ having at least one factor $X^{(2)}(k'_p)$ with $k'_p> k_p$. There are only finite number of monomials $b'$ which do not annihilate $v_{3\Lambda_0}$.   

If a quasi-particle monomial $b$ has a factor $X^{(2)}(k_p)$, with degree $k_p$ such that $k_p > k_{p+1} -12$, for some $1 \leq p \leq t$, then Lemma \ref{lemma_45} implies that $bv_{3\Lambda_0}$ can be expressed as a linear combination of quasi-particle monomial vectors $b'v_{3\Lambda_0}$, where $b'$ and $b$ have equal color-type and total degree, but $b' > b$ with respect to the linear order and as a linear combination of monomial vectors $b'v_{3\Lambda_0}$ with monomials $b'$ of greater color-type than $b$ with respect to the reverse lexicographic order.

If a quasi-particle monomial $b$ has a factor $X^{(1)}(j_p)$, with degree $j_p$ such that $j_p > -2 -6t$, for some $1 \leq p \leq s$, then, by Lemma \ref{lemma_46}, follows that $bv_{3\Lambda_0}$ can be expressed as a linear combination of quasi-particle monomial vectors $b'v_{3\Lambda_0}$ with $b’$ and $b$ having the same total degree and charge-type and $b'$ having at least one factor $X^{(1)}(j'_p)$ with $j'_p> j_p$ for some $1\leq p \leq s$ and as a linear combination of monomial vectors $b'v_{3\Lambda_0}$ with monomials $b'$ of smaller total charge than $b$. 

Finally, if a quasi-particle monomial $b$ has a factor $X^{(1)}(j_p)$, with degree $j_p$ such that $j_p > j_{p+1} -4$, for some $1 \leq p \leq s$, then by Lemma \ref{lemma_41} we may express $bv_{3\Lambda_0}$ as a linear combination of greater monomial vectors and, by induction, we may omit it from our spanning set. 

From the above analysis follows the spanning property of the set in  Theorem \ref{main_thm}. 

\subsection{Linear independence}\label{section_0503}
To finalize the proof of Theorem \ref{main_thm}, it remains to show that the corresponding spanning set is linearly independent.
This will be proved by demonstrating that the number of vectors of degree $n$ in the spanning set is given by the   $n$-th coefficient in  the right-hand side of   formula   (7) from  \cite[Theorem 10]{Kur} (see also formula (88) in \cite{KR}),
\beq\label{Kurs}
\frac{1}{(q,q^5;q^6)(q^2,q^3,q^9,q^{10};q^{12})} =F\cdot \left. \sum_{n_1,n_2\geq 0} \frac{q^{2n_1^2+6n_1n_2+6n_2^2}x^{n_1+2n_2}}{(q;q)_{n_1}(q^3;q^3)_{n_2}} \ \right|_{x=1},
\eeq
where $F=(q,q^5;q^6)^{-1}$. The product side in \eqref{Kurs} is a principally specialized character formula of standard module $L(3\Lambda_0)$ (cf. \cite{Capp1,Capp2,Capp3}).
Let us explain the correspondence between the monomial vectors from the spanning set and the coefficients of the sum side in \eqref{Kurs}.

By identifying monomials $\alpha(-i_1)\cdots \alpha(-i_r)$ with partitions $i_1+\cdots +i_r$, follows that the number of monomials
$$\alpha(i_1)\cdots \alpha(i_r), \  \ i_1+\cdots +i_r=n,$$
satisfying 
$$r \geq 0,\quad i_1\leq i_2\leq \ldots \leq i_r< 0,\quad i_p \equiv \pm 1\mod 6$$
equals to the $n$-th coefficient in
\begin{eqnarray}\nonumber
F=(q,q^5;q^6)^{-1}&=&\prod_{i \equiv \pm 1 \ \text{mod} \ 6}(1-q^i)^{-1}\\
\nonumber
&=&\prod_{i \geq 1}(1+q^{6i-5}+q^{2(6i-5)}+\cdots)(1+q^{6i-1}+q^{2(6i-1)}+\cdots).
\end{eqnarray}

For  fixed color-type $(n_1,n_2)$, we rewrite the difference conditions \eqref{diff_cond} of the spanning set as 
\begin{eqnarray}\nonumber
\sum_{p=1}^{n_1}(2+4(p-1)+6n_2)&=&2n_1^2+6n_1n_2,\\
\nonumber
\sum_{p=1}^{n_2}(6+12(p-1))&=&6n_2^2.\end{eqnarray}
For $n \in  \mathbb{N}$, set $(q^a;q^a)_n=(1- q^a)\cdots (1- q^{na})$, where $a = 1$ or $a = 3$. We have 
$$(q^a;q^a)_n^{-1}=\sum_{i \geq 0}p_n(i)q^{ai}.$$
Here $p_n(i)$ denotes the number of partition of $ia$ with at most $n$ parts (cf. \cite{A}).

From this follows that for the quasi-particle monomial 
$$
X^{(1)}(j_1)\cdots X^{(1)}(j_s)\,
X^{(2)}(k_1)\cdots X^{(2)}(k_t) 
$$
of color-type $(n_1,n_2)=(s,t)$ the term 
\beq\label{KursAG}
\frac{q^{2n_1^2+6n_1n_2+6n_2^2}x^{n_1+2n_2}}{(q;q)_{n_1}(q^3;q^3)_{n_2}}
\eeq
counts products of $n_1$ quasi-particles of charge 1 and $n_2$ quasi-particles of charge 2 which satisfy the conditions \eqref{2_cond}--\eqref{init_cond}. Hence, linear independence of the spanning set from   Theorem \ref{main_thm} follows.

\section*{Acknowledgment} 
This work has been partially supported by the Croatian Science Foundation under the project IP-2025-02-4720 and by the project ``Implementation of cutting-edge research and its application as part of the Scientific Center of Excellence for Quantum and Complex Systems, and Representations of Lie Algebras'', Grant No. PK.1.1.10.0004, co-financed by the European Union through the European Regional Development Fund - Competitiveness and Cohesion Programme 2021-2027. This research was supported by the European Union -- NextGenerationEU through the National Recovery and Resilience Plan 2021-2026 Institutional grant of University of Zagreb Faculty of Science (IK IA 1.1.3. Impact4Math).

\linespread{1.0}

\end{document}